\documentclass{amsart}
\usepackage{amsmath, amssymb, amsthm}
\usepackage{enumerate}
\usepackage{graphicx}
\usepackage{hyperref}
\usepackage{pgfplots}
\pgfplotsset{compat=1.18}

\newtheorem{theorem}{Theorem}
\newtheorem{corollary}[theorem]{Corollary}
\newtheorem{proposition}[theorem]{Proposition}
\newtheorem{lemma}[theorem]{Lemma}
\theoremstyle{definition}
\newtheorem{example}[theorem]{Example}
\theoremstyle{remark}
\newtheorem{remark}[theorem]{Remark}

\title{The aspect ratio of the Twin Dragon is $1/\varphi$}

\author{Dmitry Mekhontsev}
\address{Novosibirsk State University, Novosibirsk, Russia}
\email{mekhontsev@gmail.com}
\urladdr{https://orcid.org/0009-0001-5065-7890}
\thanks{Preprint: \href{https://arxiv.org/abs/2604.05010}{arXiv:2604.05010}.
  Zenodo: \href{https://doi.org/10.5281/zenodo.19440753}{doi:10.5281/zenodo.19440753}.}


\date{\today}

\begin{document}

\begin{abstract}
We show that the geometric aspect ratio of the Twin Dragon
--- defined via the Gaussian integer $1+i$ ---
equals $1/\varphi$, where $\varphi=(1+\sqrt{5})/2$ is the golden ratio,
and that its major axis makes angle $-\arctan\varphi$ with the horizontal.
More generally, for every equal-contraction IFS
$\{f_k(z)=a(s_k z+t_k)\}$ with rotations $s_k\in S^1$
(or reflections) and centred translations,
we give closed-form formulas for the aspect ratio, ellipticity, and
principal-axis angle of the attractor.
When all rotations are trivial, the aspect ratio depends only on~$a$
and a single anisotropy parameter $\kappa\in[0,1]$;
for $a=1/\lambda$ with $\lambda$ in an imaginary quadratic ring,
its square lies in the quadratic field $\mathbb{Q}(\sqrt{d}\,)$
where $d$ is the square-free part of $\mathbf{N}(\lambda^2\!-\!1)$.
Since $1/\varphi$ is the reciprocal of the first metallic mean, it is natural to ask
which others arise: every metallic ratio
$\mu_m=(m+\sqrt{m^2+4})/2$ $(m\geq 1)$
arises as the reciprocal aspect ratio of a plane-filling tile
over~$\mathbb{Z}[i]$;
moreover, $\mathbb{Z}[i]$ is the unique imaginary quadratic ring
where collinear digits can produce metallic-ratio aspect ratios.
\end{abstract}

\maketitle

\section{Introduction}

The Twin Dragon is a classical plane-filling fractal \cite{DavisKnuth1970} governed
by the Gaussian integer $1+i$: its two defining contractions share the linear part
$(1-i)/2$, and their translations are symmetric about the origin.
Despite this simple, purely arithmetic construction, the attractor has a non-trivial
shape that is not immediately obvious from the definition.
Visually the Twin Dragon looks somewhat elongated --- neither a square nor a thin
sliver --- which raises the question of its precise aspect ratio.

A natural quantitative measure of shape for a planar set is the \emph{geometric aspect
ratio}: the ratio of standard deviations along the principal axes of its area
distribution, or equivalently the square root of the ratio of the eigenvalues of the
second-moment (covariance) matrix $M$ of the uniform measure on the attractor.
This definition gives the expected answers for elementary shapes:
an ellipse with semi-axes $\alpha \geq \beta > 0$ has aspect ratio $\beta/\alpha$
(eigenvalues $\alpha^2/4$ and $\beta^2/4$, so the definition directly recovers the
ratio of the semi-axes); a $p \times q$ rectangle has aspect ratio $p/q$
(eigenvalues $p^2/12$ and $q^2/12$); and a square has aspect ratio~$1$.
The covariance matrix is the unique second-order shape descriptor of a
probability distribution: any similarity-invariant quantity derived
from second moments alone is a function of the eigenvalue ratio of~$M$.
Other notions of aspect ratio --- the bounding-box ratio, the Feret
diameter ratio, the eccentricity of the convex hull --- require
knowledge of the boundary or its extreme points,
which for general IFS is not available in closed form.
The covariance definition, by contrast, is determined entirely by the
invariant measure and is exactly computable for self-similar IFS via
the moment fixed-point equation.

The approach of characterising IFS attractors through moments of their invariant
measures was introduced by Vrscay and Roehrig \cite{VrscayRoehrig1989} and
systematically developed in \cite{Vrscay1991}.
Further analysis of self-similar measures via moments and Fourier transforms
appears in Strichartz~\cite{Strichartz1993} and Lau--Ngai~\cite{LauNgai1999};
however, explicit aspect-ratio formulas were not derived in those works.
The key observation is that the
self-similar measure satisfies a linear fixed-point equation in its moments, which
can be solved exactly.  When the open set condition holds, Hutchinson's theorem
\cite{Hutchinson1981} guarantees that the self-similar measure coincides with the
normalised Hausdorff measure $\mathcal{H}^s|_A$, so the covariance matrix captures
the actual geometry of the attractor at whatever dimension $s$ it carries.  In the
special case where the similarity dimension $s$ equals the ambient dimension,
$\mathcal{H}^s$ reduces to Lebesgue area.

We apply this method to the Twin Dragon.  The covariance matrix turns out to be
$M = \tfrac{1}{5}\begin{pmatrix}2&-1\\-1&3\end{pmatrix}$,
whose eigenvalues involve $\sqrt{5}$, leading to the aspect ratio $1/\varphi$,
where $\varphi = (1+\sqrt{5})/2$ is the golden ratio.  The appearance of $\varphi$
in a fractal with no pentagonal geometry is unexpected.
Although the method of moments for IFS goes back to \cite{VrscayRoehrig1989},
the closed-form aspect ratio and orientation formulas
(Theorem~\ref{thm:family} below) and the
characterisation of the quadratic field containing its square for Gaussian-integer parameters
(Corollary~\ref{cor:field}) are, to our knowledge, new.

\textit{Discovery.}
The result was first found empirically using the IFStile software
\cite{Mekhontsev2019}, which implements the moment invariants of
\cite{VrscayRoehrig1989,Vrscay1991} as part of its algebraic IFS search engine.
Among the similarity invariants reported by IFStile for the Twin Dragon, the
aspect ratio $1/\varphi$ appeared as a numerical value; the present paper supplies
the exact proof and its generalisation to arbitrary
equal-contraction IFS with centred translations.
This is an instance of \emph{experimental mathematics} \cite{BorweinBailey2004}:
a computer experiment suggests a precise conjecture, which is then confirmed
by a rigorous proof.

The paper is organised as follows.
Section~\ref{sec:family} derives closed-form formulas for the
aspect ratio and principal-axis angle
(Theorem~\ref{thm:family}) and classifies the quadratic field
of the squared aspect ratio (Corollary~\ref{cor:field}).
Section~\ref{sec:family-N} applies the formula to a two-parameter family of
plane-filling tiles parameterised by $x^2-\ell x+N$
(Proposition~\ref{prop:family-N}).
Section~\ref{sec:metallic} specialises to tiles over~$\mathbb{Z}[i]$,
proves that every metallic ratio $\mu_m$ arises as
the reciprocal aspect ratio of such a tile
(Theorem~\ref{thm:metallic}), and shows that $\mathbb{Z}[i]$ is the unique imaginary
quadratic ring where collinear tiles achieve metallic-ratio aspect ratios
(Corollary~\ref{cor:unique-Zi}).
The eponymous Twin Dragon is treated as Example~\ref{ex:twin-dragon}
immediately after Corollary~\ref{cor:field}.

\section{The aspect ratio formula}\label{sec:family}

An \emph{iterated function system} (IFS) is a finite collection
of contractions $f_1,\ldots,f_N$ on a complete metric space.
By Hutchinson's theorem \cite{Hutchinson1981} there exists a unique
non-empty compact set~$A$, the \emph{attractor}, satisfying
$A = f_1(A) \cup \cdots \cup f_N(A)$.
For background see \cite{Barnsley1993, Falconer2014}; for self-affine digit tiles
in particular, see \cite{Bandt1991, Kenyon1992, LagariasWang1996}.

Consider an equal-weight IFS of $N\geq 2$ contractions on~$\mathbb{C}$,
each of the form
\[
  f_k(z) = a(s_k z + t_k)
  \qquad\text{or}\qquad
  f_k(z) = a(s_k\overline{z} + t_k),
\]
where $a \in \mathbb{C}$, $|a| < 1$, the rotations satisfy
$s_k \in S^1$, and $t_1,\ldots,t_N \in \mathbb{C}$ satisfy
$\sum_{k=1}^N t_k = 0$.

\begin{remark}[Centering]\label{rem:centering}
The condition $\sum t_k = 0$ is no loss of generality.
The translation $z \mapsto z - c$ replaces each~$t_k$ by
$\tilde{t}_k = s_k c + t_k - c/a$ (preserving maps) or
$\tilde{t}_k = s_k \bar{c} + t_k - c/a$ (reversing maps).
Setting $\sum \tilde{t}_k = 0$ gives a linear equation in~$c$
(and~$\bar{c}$ when reversing maps are present)
whose unique solvability is guaranteed by $|a| < 1$.
When all~$s_k$ coincide, this simplifies to
$t_k \mapsto t_k - \bar{t}$, a uniform shift of the digits.
\end{remark}

Define the \emph{digit variance} $\sigma^2 := \tfrac{1}{N}\sum_k |t_k|^2 > 0$,
the \emph{algebraic variance} $\tau := \tfrac{1}{N}\sum_k t_k^2 \in \mathbb{C}$,
and the \emph{digit anisotropy}
\[
  \kappa := \frac{|\tau|}{\sigma^2}
         = \frac{|\sum t_k^2|}{\sum |t_k|^2} \in [0,1].
\]
The bound $\kappa \leq 1$ follows from the triangle inequality
$|\sum t_k^2| \leq \sum |t_k^2| = \sum |t_k|^2$.
When all $t_k$ are collinear ($t_k = e^{i\theta}r_k$ with $r_k \in \mathbb{R}$),
one has $\kappa = 1$; vertices of a regular $N$-gon ($N\geq 3$) give $\kappa = 0$.
The Twin Dragon is the case $N=2$, $a = (1-i)/2$, $s_k = 1$, $t_k = \pm 1$
($\sigma^2=1$, $\kappa=1$), treated in Example~\ref{ex:twin-dragon} below.

\begin{theorem}\label{thm:family}
With the notation above, let $I$ and~$J$ denote the sets of
orientation-preserving and -reversing indices, and define
\[
  u = \frac{1}{N}\!\sum_{k \in I}s_k^2,\qquad
  v = \frac{1}{N}\!\sum_{k \in J}s_k^2.
\]
Then the second moment $\omega := E_\mu[Z^2]$ is
\begin{equation}\label{eq:omega}
  \omega
  = \frac{a^2\bigl(\tau + \overline{a^2}(v\,\overline{\tau}
          - \bar{u}\,\tau)\bigr)}
         {|1-a^2 u|^2 - |a|^4|v|^2},
\end{equation}
the aspect ratio is
$\mathrm{AR}^2 = (1-\rho)/(1+\rho)$, where
\begin{equation}\label{eq:rho}
  \rho = \frac{(1-|a|^2)\,|\omega|}{|a|^2\,\sigma^2},
\end{equation}
and the major axis of the attractor makes angle
\begin{equation}\label{eq:psi}
  \psi = \tfrac{1}{2}\arg\omega
\end{equation}
with the positive $x$-axis.
When $v = 0$ \textup{(}all maps orientation-preserving\textup{)},
\eqref{eq:omega} reduces to
$\omega = a^2\tau/(1-a^2 u)$, and
\begin{equation}\label{eq:rho-pres}
  \rho = \frac{(1-|a|^2)\,|\tau|}{|1-a^2 u|\,\sigma^2}.
\end{equation}
In general, the aspect ratio depends on $u$, $v$, and the
complex value of~$\tau$ \textup{(}not just~$|\tau|$\textup{)}.
\end{theorem}

\begin{proof}
Let $Z$ be distributed according to the invariant measure $\mu$.
The fixed-point equation
\begin{equation}\label{eq:fp}
  E[h(Z)] = \frac{1}{N}\!\Bigl(
    \sum_{k\in I} E\bigl[h\bigl(a(s_k Z + t_k)\bigr)\bigr]
    + \sum_{k\in J} E\bigl[h\bigl(a(s_k\overline{Z} + t_k)\bigr)\bigr]
  \Bigr)
\end{equation}
applied to $h(z)=z$ gives
$E[Z] = a\bigl(u'\,E[Z] + v'\,\overline{E[Z]}
  + \tfrac{1}{N}\!\sum t_k\bigr)$
where $u' = \frac{1}{N}\sum_{k\in I}s_k$,
$v' = \frac{1}{N}\sum_{k\in J}s_k$;
since $|a|(|u'|+|v'|)\leq|a|<1$ and $\sum t_k = 0$,
the unique solution is $E[Z]=0$.

Applying \eqref{eq:fp} to $h(z)=|z|^2$:
since $E[Z]=0$, the cross terms $2\operatorname{Re}(\overline{s_k Z}\,t_k)$
vanish in expectation, and $|s_k z|^2 = |z|^2$, giving
\[
  E[|Z|^2] = |a|^2\bigl(E[|Z|^2] + \sigma^2\bigr),
\]
so $E[|Z|^2] = |a|^2\sigma^2/(1-|a|^2)$.

Applying \eqref{eq:fp} to $h(z)=z^2$ and writing $\omega := E[Z^2]$:
$(s_k z)^2 = s_k^2 z^2$ for preserving maps and
$s_k^2\overline{z}^2$ for reversing ones, while the cross
terms $2s_k z\cdot t_k$ vanish in expectation
and the translations contribute $\sum t_k^2 = N\tau$, so
\begin{equation}\label{eq:omega-fp}
  \omega = a^2\bigl(u\,\omega + v\,\overline{\omega} + \tau\bigr).
\end{equation}
Rewriting as the $\mathbb{R}$-linear system
$(1-a^2 u)\,\omega - a^2 v\,\overline{\omega} = a^2\tau$
and applying Cramer's rule
($\Delta := |1-a^2 u|^2 - |a|^4|v|^2 > 0$
since $|a|^2(|u|+|v|)\leq |a|^2 < 1$)
gives~\eqref{eq:omega}.
The ellipticity
$\rho = |\omega|/E[|Z|^2]$ gives~\eqref{eq:rho};
since the covariance matrix $M$ of the real 2D vector
$(\mathrm{Re}\,Z,\mathrm{Im}\,Z)$ is positive semidefinite,
$|\omega|\leq E[|Z|^2]$ and $\rho\in[0,1]$, so
$\mathrm{AR}^2 = (1-\rho)/(1+\rho)\in[0,1]$.
The orientation~\eqref{eq:psi} follows from the
covariance decomposition in Remark~\ref{rem:orientation}.

When $v=0$, the equation is $\mathbb{C}$-linear:
$\omega = a^2\tau/(1-a^2 u)$,
giving~\eqref{eq:rho-pres}.
\end{proof}

When all rotations are trivial, the formulas simplify to depend
on a single anisotropy parameter.

\begin{corollary}\label{cor:kappa}
When all $s_k = 1$ \textup{(}so $u=1$, $v=0$\textup{)},
the aspect ratio depends on the translations only through
$\kappa = |\tau|/\sigma^2 \in [0,1]$:
\begin{equation}\label{eq:family}
  \mathrm{AR}^2 = \frac{|1-a^2| - (1-|a|^2)\,\kappa}{|1-a^2| + (1-|a|^2)\,\kappa}.
\end{equation}
\begin{enumerate}[\upshape(a)]
\item For fixed~$a$, $\mathrm{AR}$ is strictly decreasing in~$\kappa$:
$\kappa=0$ \textup{(}isotropic digit set\textup{)} gives $\mathrm{AR}=1$,
and $\kappa=1$ \textup{(}collinear translations\textup{)} gives the
most elongated attractor.
\item $\mathrm{AR}=0$ if and only if $a \in \mathbb{R}$ and $\kappa=1$.
\end{enumerate}
For collinear translations \textup{(}$\kappa=1$,
$\tau>0$ after aligning the digit axis with~$\mathbb{R}$\textup{)},
\begin{equation}\label{eq:psi-simple}
  \tan 2\psi = \frac{\operatorname{Im}(a^2)}{\operatorname{Re}(a^2)-|a|^4}\,;
\end{equation}
for non-collinear digits, $\psi$ is shifted
by~$\tfrac{1}{2}\arg\tau$.
\end{corollary}

\begin{proof}
Setting $s_k=1$ gives $u=1$, $v=0$, and~\eqref{eq:rho-pres}
yields $\rho = (1-|a|^2)\kappa/|1-a^2|$,
giving~\eqref{eq:family}.
Properties~(a)--(b) follow:
$\rho$ is increasing in~$\kappa$,
vanishes at $\kappa=0$,
and $\rho=1$ (i.e.\ $\mathrm{AR}=0$) requires
$\kappa=1$ and $|1-a^2|=1-|a|^2$,
which holds iff $a \in \mathbb{R}$.
\end{proof}

The aspect ratio in~\eqref{eq:family} involves $|1-a^2|$, which is
generically transcendental.  When $a = 1/\lambda$ for an algebraic
integer~$\lambda$ in an imaginary quadratic field
$K = \mathbb{Q}(\sqrt{-D})$, this modulus reduces to
$\sqrt{\mathbf{N}(\lambda^2-1)}/\mathbf{N}(\lambda)$
(where $\mathbf{N}$ is the field norm,
$\mathbf{N}(\alpha)=|\alpha|^2\in\mathbb{Z}$ for
$\alpha\in\mathcal{O}_K$),
and the aspect ratio becomes a quadratic irrational.

\begin{corollary}\label{cor:field}
Let $K = \mathbb{Q}(\sqrt{-D})$ be an imaginary quadratic field
with ring of integers~$\mathcal{O}_K$,
let $\lambda \in \mathcal{O}_K$ with $|\lambda|>1$ and $\lambda\notin\mathbb{R}$,
and set $a = 1/\lambda$.  For any IFS as in Corollary~\textup{\ref{cor:kappa}}
with $\kappa=1$,
the squared aspect ratio~\eqref{eq:family} lies
in $\mathbb{Q}(\sqrt{d}\,)$,
where $d$ is the square-free part of $\mathbf{N}(\lambda^2-1)$.
Explicitly,
\begin{equation}\label{eq:field}
  \mathrm{AR}^2
  = \frac{\sqrt{\mathbf{N}(\lambda^2-1)} - (\mathbf{N}(\lambda)-1)}
         {\sqrt{\mathbf{N}(\lambda^2-1)} + (\mathbf{N}(\lambda)-1)}.
\end{equation}
\end{corollary}

\begin{proof}
For $\lambda\in\mathcal{O}_K$,
$\mathbf{N}(\lambda)=|\lambda|^2\in\mathbb{Z}$
and $\mathbf{N}(\lambda^2-1)=|\lambda^2-1|^2\in\mathbb{Z}$.
Substitute $|a|^2 = 1/\mathbf{N}(\lambda)$,
$|1-a^2| = \sqrt{\mathbf{N}(\lambda^2-1)}/\mathbf{N}(\lambda)$,
and $\kappa=1$
into~\eqref{eq:family}.
\end{proof}

\begin{example}[Twin Dragon]\label{ex:twin-dragon}
The Twin Dragon $A\subset\mathbb{C}$ is the attractor of the IFS
\cite{DavisKnuth1970}
\[
  A = f_1(A)\cup f_2(A),\qquad
  f_{1,2}(z) = \tfrac{1-i}{2}(z\pm 1),
\]
governed by the Gaussian integer $\lambda=1+i$
\textup{(}so $a=1/\lambda=(1-i)/2$, $N=2$, $\ell=2$,
$\sigma^2=1$, $\kappa=1$\textup{)}.
This is the case $\lambda=1+i$ of Corollary~\ref{cor:field},
with $\lambda^2-1 = -1+2i$, $\mathbf{N}(\lambda^2-1)=5$, $d=5$:
substituting in~\eqref{eq:field} gives
\[
  \mathrm{AR}^2
  = \frac{\sqrt{5}-1}{\sqrt{5}+1}
  = \frac{3-\sqrt{5}}{2}
  = \frac{1}{\varphi^2},
  \qquad
  \mathrm{AR} = \frac{1}{\varphi}.
\]
The open set condition holds \cite[Corollary~3]{Bandt1991}, so the
self-similar measure coincides with $\mathcal{H}^2|_A$
\cite[\S5]{Hutchinson1981}; since the similarity dimension equals the
ambient dimension, the covariance matrix $M$ describes the actual
geometric shape of the tile \textup{(}Figure~\ref{fig:ellipses}a\textup{)}.
Explicitly, $E[Z^2]=a^2/(1-a^2)=(-1-2i)/5$,
$E[|Z|^2]=|a|^2/(1-|a|^2)=1$, whence
\[
  M = \frac{1}{5}\begin{pmatrix}2 & -1 \\ -1 & 3\end{pmatrix},
  \qquad
  I_{1,2} = \tfrac{1}{2}\bigl(1\mp\tfrac{1}{\sqrt{5}}\bigr),
\]
confirming $\sqrt{I_1/I_2}=1/\varphi$.
The principal-axis angle follows from
Proposition~\ref{prop:family-N}\textup{(}iii\textup{)}:
$\tan\psi=-\varphi$, i.e.\ $\psi=-\arctan\varphi\approx-58.3^\circ$
along the direction $(1,-\varphi)$.
\end{example}

\begin{remark}
The appearance of $\varphi$ is unexpected: no pentagon, decagon, or
Fibonacci recurrence enters the Twin Dragon's definition.
The origin of~$5$ --- and hence $\varphi$ --- is a single Gaussian
prime.  Since $1-a^2=(2+i)/2$ and $\kappa=1$, the ellipticity is
$\rho = (1-|a|^2)/|1-a^2| = 1/\sqrt{5}$, where $5=\mathbf{N}(2+i)$
is the norm of the Gaussian prime $2+i$;
thus $\varphi$ enters via the factorisation $5=(2+i)(2-i)$ in
$\mathbb{Z}[i]$, and $\mathrm{AR}^2=(1-\rho)/(1+\rho)=1/\varphi^2$
uses $\varphi^2=\varphi+1$.
\end{remark}

\begin{table}[ht]
\centering
\small
\caption{Aspect ratios for small Gaussian integers $\lambda\in\mathbb{Z}[i]$,
illustrating Corollary~\ref{cor:field}.
The quadratic field containing $\mathrm{AR}^2$ is determined by the
square-free part of $\mathbf{N}(\lambda^2-1)$.}
\label{tab:examples}
\begin{tabular}{ccccl}
$\lambda$ & $\lambda^2-1$ & $\mathbf{N}(\lambda^2-1)$ & $d$ & $\mathrm{AR}$ \\[2pt]
\hline\\[-8pt]
$1+i$  & $-1+2i$  & $5$  & $5$ & $1/\varphi = (\sqrt{5}-1)/2 \approx 0.618$ \\
$2+i$  & $2+4i$   & $20$ & $5$ & $1/\varphi^3 = \sqrt{5}-2 \approx 0.236$ \\
$1+2i$ & $-4+4i$  & $32$ & $2$ & $\sqrt{2}-1 \approx 0.414$ \\
\end{tabular}
\end{table}

\noindent
The first two rows of Table~\ref{tab:examples} share the same quadratic
field $\mathbb{Q}(\sqrt{5})$
because the rational prime~$5$ divides both norms.
This is not a coincidence: the Cayley-type transform~\eqref{eq:field}
maps $\sqrt{\mathbf{N}(\lambda^2-1)}$ to an element of the
real quadratic field $\mathbb{Q}(\sqrt{d}\,)$, and the particular
metallic ratio that appears is controlled by the arithmetic of this field.
In the first row, $\sqrt{5}$ yields a rational expression in $\varphi$
thanks to the identity $\varphi^2 = \varphi+1$: the golden ratio is the
fundamental unit of the ring $\mathbb{Z}[\varphi]$, so every element of
$\mathbb{Q}(\sqrt{5})$ in $(0,1)$ of the form
$(\sqrt{5}-n)/(\sqrt{5}+n)$ is a power of $1/\varphi$.
The third row gives the \emph{silver ratio} $\delta_S = \sqrt{2}-1$
(fundamental unit of $\mathbb{Z}[\sqrt{2}]$).

All examples in Table~\ref{tab:examples} have $s_k=1$.
When rotations are present, the general formula~\eqref{eq:rho-pres}
applies, and the aspect ratio can change dramatically.

\begin{example}[L\'evy C curve]\label{ex:levy}
Take the same contraction $a = (1-i)/2$ and centred digits
$t_0 = -1$, $t_1 = 1$ as in the Twin Dragon,
but set~$s_0 = i$, $s_1 = 1$
(so $f_0$ includes an extra $90^\circ$ rotation).
The attractor is a copy of the \emph{L\'evy C curve}.
Since $\sigma^2 = 1$, $\tau = 1$, $\kappa = 1$,
the digit geometry is identical to the Twin Dragon;
only the rotation average changes:
$u = (i^2 + 1^2)/2 = 0$.
By~\eqref{eq:rho-pres},
\[
  \rho
  = \frac{(1-|a|^2)\,|\tau|}{|1-a^2 u|\,\sigma^2}
  = \frac{\tfrac{1}{2}\cdot 1}{|1-0|\cdot 1}
  = \frac{1}{2},
\]
whence $\mathrm{AR}^2 = (1-\tfrac{1}{2})/(1+\tfrac{1}{2}) = 1/3$,
i.e.\ $\mathrm{AR} = 1/\sqrt{3} \approx 0.577$.
For comparison, the Twin Dragon has $u=1$ and
$\mathrm{AR} = 1/\varphi \approx 0.618$
(Table~\ref{tab:examples}).
Thus a single rotation $s_0 \mapsto i$ replaces $u=1$ with $u=0$,
changing the aspect ratio from $1/\varphi$ to $1/\sqrt{3}$.
More generally, for $N=2$ one has $u = 0$ if and only if
$s_1/s_0 = \pm i$
(the two rotations differ by~$90^\circ$).
The \emph{Heighway Dragon}, with $s_0 = -1$ and $s_1 = i$,
is another instance: its attractor has a very different shape
but the same $\mathrm{AR}^2 = 1/3$.
\end{example}

\begin{example}[Conjugate parallelogram]\label{ex:refl}
Replace the two preserving maps of the Twin Dragon
by their orientation-reversing counterparts:
$f_k(z) = a(\overline{z} \pm 1)$, $a = (1-i)/2$.
The digits $t_k = \pm 1$ and contraction~$a$ are unchanged,
so $\sigma^2 = 1$, $\tau = 1$, $\kappa = 1$;
but now $I = \emptyset$, $J = \{1,2\}$, giving
$u = 0$, $v = 1$.
Setting $u=0$, $v=1$, $\tau=1$ in~\eqref{eq:omega} gives
\[
  \omega
  = \frac{a^2(1+\overline{a^2})}{1-|a|^4}
  = \frac{(-i/2)(1+i/2)}{3/4}
  = \frac{1-2i}{3}\,,
\]
so $|\omega| = \sqrt{5}/3$ and $\rho = \sqrt{5}/3$.
Therefore
\[
  \mathrm{AR}^2
  = \frac{3-\sqrt{5}}{3+\sqrt{5}}
  = \frac{7-3\sqrt{5}}{2}
  = \frac{1}{\varphi^4}\,,
  \qquad
  \mathrm{AR} = \frac{1}{\varphi^2} = \frac{3-\sqrt{5}}{2} \approx 0.382.
\]
The attractor is a parallelogram
with vertices $\pm i$ and $\pm(2-i)$
(one verifies that both maps permute these four vertices).
It is more elongated than the Twin Dragon ($1/\varphi^2 < 1/\varphi$).
Thus replacing $z$ by $\overline{z}$ in both maps
(switching from $u=1$, $v=0$ to $u=0$, $v=1$)
changes the aspect ratio from $1/\varphi$ to $1/\varphi^2$,
while keeping it in $\mathbb{Q}(\sqrt{5})$.
\end{example}

\begin{figure}[ht]
\centering
\begin{minipage}[b]{0.45\textwidth}\centering
  \includegraphics[width=\textwidth]{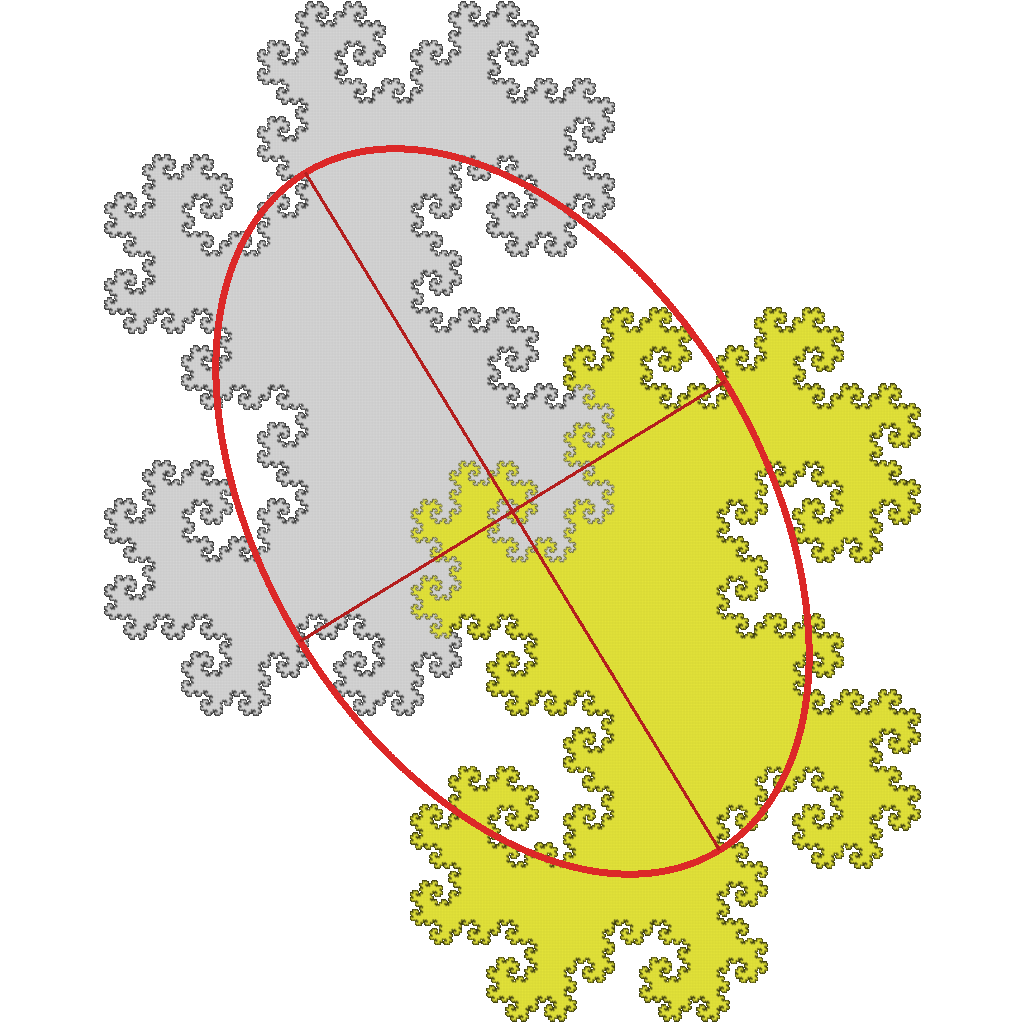}\\[2pt]
  (a) Twin Dragon
\end{minipage}\hfill
\begin{minipage}[b]{0.45\textwidth}\centering
  \includegraphics[width=\textwidth]{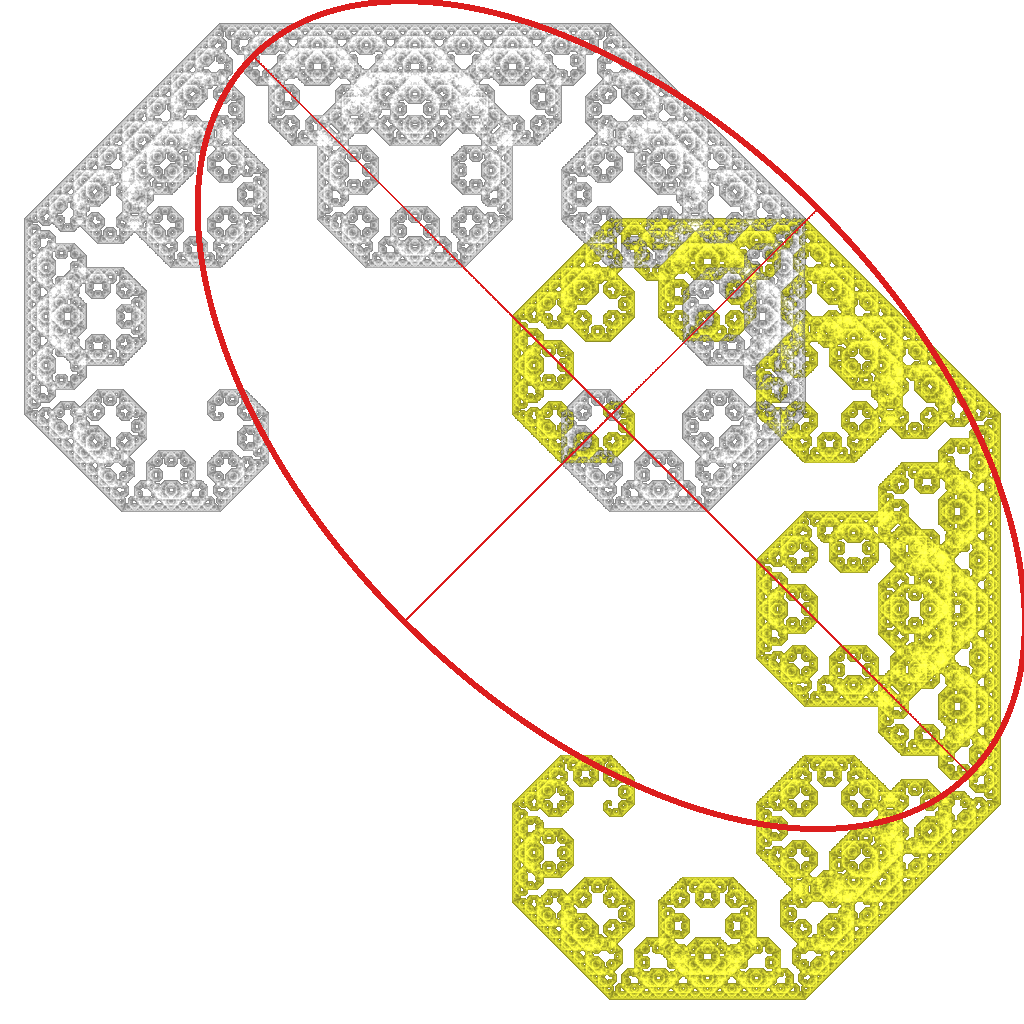}\\[2pt]
  (b) L\'evy C curve
\end{minipage}\\[8pt]
\begin{minipage}[b]{0.45\textwidth}\centering
  \includegraphics[width=\textwidth]{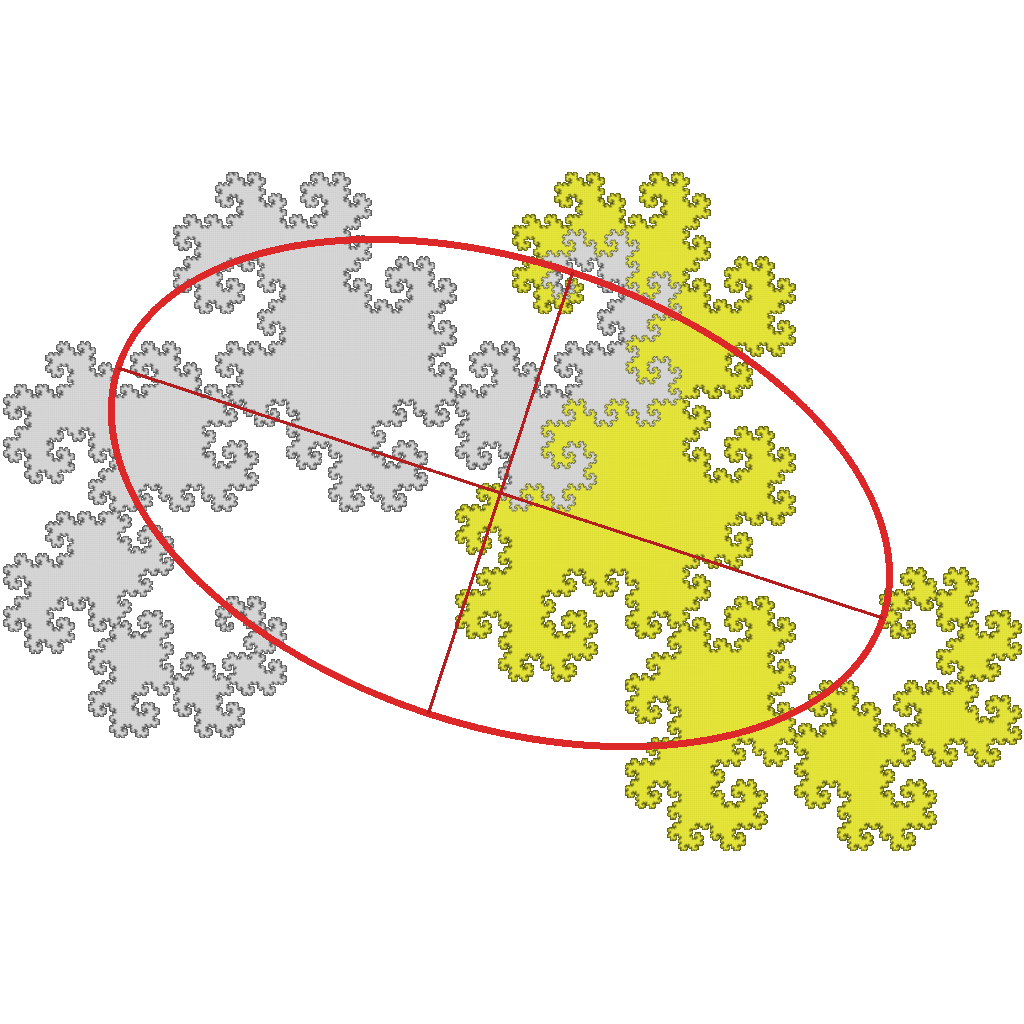}\\[2pt]
  (c) Heighway Dragon
\end{minipage}\hfill
\begin{minipage}[b]{0.45\textwidth}\centering
  \includegraphics[width=\textwidth]{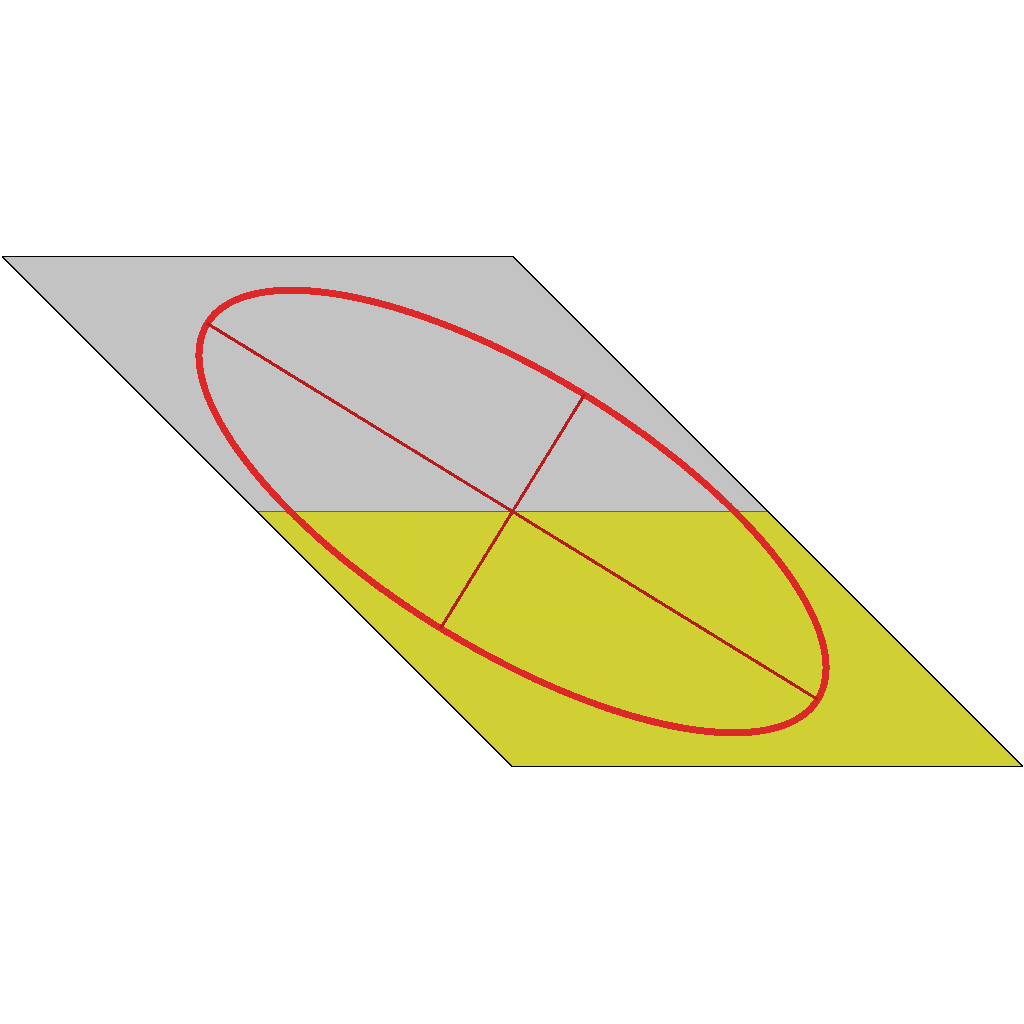}\\[2pt]
  (d) Conjugate parallelogram
\end{minipage}
\caption{Four IFS with contraction $a = (1-i)/2$ and centred digits
in~$\mathbb{Z}[i]$, shown with their $1.5\sigma$ covariance ellipses.
The Twin Dragon~(a) has $s_k=1$ and $\mathrm{AR}=1/\varphi$.
The L\'evy~C curve~(b) and Heighway Dragon~(c) both have $u=0$ and
$\mathrm{AR}=1/\sqrt{3}$: the ellipses share the same
aspect ratio but have different orientations
($\psi=0^\circ$ and $\psi\approx 18^\circ$, respectively).
The Conjugate parallelogram~(d) has all maps orientation-reversing
($u=0$, $v=1$) and $\mathrm{AR}=1/\varphi^2$.}
\label{fig:ellipses}
\end{figure}

\smallskip
Corollary~\ref{cor:field} assumes $\kappa=1$.
For arbitrary algebraic digits the anisotropy may change the quadratic
field, as the following variant records.

\begin{corollary}\label{cor:field-gen}
For any IFS $\{f_k(z)=a(z+t_k)\}$ with $a=1/\lambda$,
$\lambda\in\mathcal{O}_K\setminus\mathbb{R}$
\textup{(}$K$ an imaginary quadratic field\textup{)},
and digits $t_k\in\mathcal{O}_K$,
the anisotropy satisfies $\kappa^2\in\mathbb{Q}$.
If $\kappa>0$, then
\[
  \mathrm{AR}^2 \in
  \mathbb{Q}\!\left(\sqrt{\frac{\mathbf{N}(\lambda^2-1)}{\kappa^2}}\right).
\]
Equivalently, if $r=\kappa^2$ and $d_r$ is the square-free part of
$\mathbf{N}(\lambda^2-1)/r$, then
$\mathrm{AR}^2\in\mathbb{Q}(\sqrt{d_r}\,)$.  For $\kappa=0$ one has
$\mathrm{AR}=1$.
\end{corollary}

\begin{proof}
Let $\bar t = N^{-1}\sum_k t_k$.  Since $t_k\in K$ and $\bar t\in K$,
we have $\tau = \sum(t_k-\bar t)^2 \in K$ and
$\sigma^2 = \sum |t_k-\bar t|^2 \in \mathbb{Q}$.
Thus $\kappa^2 = |\tau|^2/\sigma^4 = \mathbf{N}(\tau)/\sigma^4\in\mathbb{Q}$.
Substituting $|a|^2=1/N_0$ and
$|1-a^2|=\sqrt{M}/N_0$, where
$N_0=\mathbf{N}(\lambda)$ and $M=\mathbf{N}(\lambda^2-1)$, gives
\[
  \mathrm{AR}^2
  = \frac{\sqrt{M}-(N_0-1)\kappa}
         {\sqrt{M}+(N_0-1)\kappa}.
\]
For $\kappa>0$, divide numerator and denominator by~$\kappa$:
\[
  \mathrm{AR}^2
  = \frac{\sqrt{M/\kappa^2}-(N_0-1)}
         {\sqrt{M/\kappa^2}+(N_0-1)},
\]
which lies in $\mathbb{Q}(\sqrt{M/\kappa^2}\,)$ because
$M/\kappa^2\in\mathbb{Q}$.  If $\kappa=0$, then
\eqref{eq:family} gives $\mathrm{AR}=1$.
\end{proof}

\begin{remark}
For collinear translations ($\kappa=1$),
formula~\eqref{eq:family} shows that $\mathrm{AR}$ depends on $a$ only
through $(|a|^2, |1-a^2|)$: the attractor is nearly round when $|1-a^2|$
is large relative to $1-|a|^2$, and degenerates to a segment when
$|1-a^2|\approx 1-|a|^2$ (i.e.\ $a$ nearly real).
For every $a$ in the family, $\mathrm{AR}(a) \leq |a|$, with equality
if and only if $a$ is purely imaginary
(the attractor is then a rectangle, cf.\ \S\ref{sec:family-N}).
Indeed, $|1-a^2| \leq 1+|a|^2$ (triangle inequality), with equality iff
$a^2/|a^2|=-1$, i.e.\ $a\in i\mathbb{R}$.
Substituting $|1-a^2|=1+|a|^2$ into~\eqref{eq:family} gives
$\mathrm{AR}^2 = |a|^2$.
Figure~\ref{fig:ar} shows $\mathrm{AR}$ over the entire sub-family $|a|=1/\sqrt{2}$.
\end{remark}

\begin{remark}\label{rem:orientation}
The orientation formula~\eqref{eq:psi} follows from the decomposition
of the covariance matrix.  Writing $\omega = |\omega|\, e^{2i\psi}$
for $Z = X + iY$ with $E[Z]=0$:
\[
  M = \frac{E[|Z|^2]}{2}\,I
    + \frac{|\omega|}{2}
      \begin{pmatrix}\cos 2\psi & \sin 2\psi \\ \sin 2\psi & -\cos 2\psi\end{pmatrix},
\]
whose larger eigenvalue has eigenvector $(\cos\psi, \sin\psi)$.
For the collinear case~\eqref{eq:psi-simple},
$\operatorname{Im}(a^2)$ measures the rotation per level of the
IFS hierarchy, and $\operatorname{Re}(a^2)-|a|^4$ is the real drift;
their ratio determines the principal-axis direction.
\end{remark}

\section{The \texorpdfstring{$\lambda_{N,\ell}$}{λ\_{N,ℓ}} family}\label{sec:family-N}

A natural two-parameter family of plane-filling tiles arises
by applying Corollary~\ref{cor:kappa} to collinear digits.
For integers $N\geq 2$ and $0 \leq \ell < 2\sqrt{N}$, let
$\lambda_{N,\ell}$ be a root of $x^2-\ell x+N$
(so $|\lambda_{N,\ell}|^2=N$ and $\lambda_{N,\ell}\notin\mathbb{R}$).
The digits $\{0,1,\ldots,N\!-\!1\}$ form a complete residue system (CRS) of
$\mathbb{Z}[\lambda_{N,\ell}]$ modulo~$\lambda_{N,\ell}$
\cite{Kenyon1992, LagariasWang1996}.
Indeed, every element $a+b\lambda_{N,\ell}\in\mathbb{Z}[\lambda_{N,\ell}]$
satisfies $\lambda_{N,\ell}(a+b\lambda_{N,\ell})
= a\lambda_{N,\ell}+b(\ell\lambda_{N,\ell}-N)$
(using $\lambda_{N,\ell}^2=\ell\lambda_{N,\ell}-N$),
whose integer part is $-bN$;
hence $a+b\lambda_{N,\ell}\equiv 0\pmod{\lambda_{N,\ell}}$ forces
$b=0$ and $N\mid a$, so
$\lambda_{N,\ell}\mathbb{Z}[\lambda_{N,\ell}]\cap\mathbb{Z}=N\mathbb{Z}$
and the $N$ residues $\{0,\ldots,N\!-\!1\}$ are distinct.
The $N$-map IFS with $a=1/\lambda_{N,\ell}$ and digits centred
to $\sum t_j=0$ therefore produces a tile.
The aspect ratio and orientation of these tiles
admit explicit formulas.

\begin{proposition}\label{prop:family-N}
Let $N\geq 2$ and $0 \leq \ell < 2\sqrt{N}$, and let $\lambda = \lambda_{N,\ell}$
be a root of $x^2-\ell x+N$.
For the $N$-map tiling IFS with contraction $a = 1/\lambda$:
\begin{enumerate}[{\upshape (i)}]
\item $|\lambda^2-1|^2 = (N+1)^2 - \ell^2$, and the aspect ratio is
\begin{equation}\label{eq:family-N}
  \mathrm{AR}_{N,\ell}
  = \sqrt{\frac{\sqrt{(N+1)^2-\ell^2}\;-(N-1)}
               {\sqrt{(N+1)^2-\ell^2}\;+(N-1)}}\,.
\end{equation}
\item The major axis of inertia makes angle
\begin{equation}\label{eq:psi-N}
  \psi_{N,\ell}
  = -\tfrac{1}{2}\arg(\lambda^2-1)
\end{equation}
with the positive $x$-axis, where
$\lambda^2-1 = \tfrac{1}{2}\bigl((\ell^2\!-\!2N\!-\!2) + i\ell\sqrt{4N-\ell^2}\bigr)$.
\item When $\ell=2$ \textup{(}the sub-family
$\lambda_N = 1+i\sqrt{N\!-\!1}$\textup{)},
\eqref{eq:family-N} simplifies to
$\mathrm{AR}_N = (\sqrt{N+3}-\sqrt{N-1})/2$,
and $\tan\psi$ satisfies $t^2+t\sqrt{N\!-\!1}-1=0$.
When additionally $N\!-\!1=m^2$ for an integer $m\geq 1$, this equation
becomes $t^2+mt-1=0$, whose roots are $1/\mu_m$ and $-\mu_m$,
where $\mu_m=(m+\sqrt{m^2+4})/2$ is the $m$-th metallic mean
\textup{(}the root of $x^2-mx-1=0$\textup{)};
so $\tan\psi\in\{1/\mu_m,\,-\mu_m\}$, with the sign fixed by~\eqref{eq:psi-N}.
For $m=1$, $\tan\psi=-\varphi$; for $m=2$, $\tan\psi=-(1+\sqrt{2})$.
\end{enumerate}
\end{proposition}

\begin{proof}
Part~(i): Expanding $\lambda = (\ell+i\sqrt{4N-\ell^2})/2$ gives
\[
  |\lambda^2-1|^2
  = \tfrac{1}{4}\bigl[(\ell^2\!-\!2N\!-\!2)^2 + \ell^2(4N-\ell^2)\bigr]
  = (N+1)^2 - \ell^2.
\]
Substituting $|a|^2 = 1/N$ and $|1-a^2| = \sqrt{(N+1)^2-\ell^2}/N$
into~\eqref{eq:family} yields~\eqref{eq:family-N}.
Part~(ii): for collinear digits $\omega = a^2\sigma^2/(1-a^2)
= \sigma^2/(\lambda^2-1)$, so~\eqref{eq:psi} gives
$\psi = -\tfrac{1}{2}\arg(\lambda^2-1)$.
Part~(iii): Setting $\ell=2$ gives
$(N+1)^2-4 = (N-1)(N+3)$; the AR formula rationalises to
$(\sqrt{N+3}-\sqrt{N-1})/2$.
For the tangent, $\lambda^2-1 = (1-N) + 2i\sqrt{N\!-\!1}$;
the double-angle identity $\tan 2\psi = 2\tan\psi/(1-\tan^2\!\psi)
= 2/\sqrt{N\!-\!1}$ yields the stated equation.
\end{proof}

Selected values:

\smallskip
\centerline{\small
\begin{tabular}{ccccl}
$N$ & $\ell$ & $\lambda_{N,\ell}$ & field & $\mathrm{AR}_{N,\ell}$ \\[2pt]\hline\\[-8pt]
$2$ & $0$ & $i\sqrt{2}$ & $\mathbb{Q}$ &
  $1/\sqrt{2} \approx 0.707$ (rectangle) \\
$2$ & $1$ & $(1+i\sqrt{7})/2$ & $\mathbb{Q}(\sqrt{2})$ &
  $\approx 0.691$ (tame twindragon) \\
$2$ & $2$ & $1+i$ & $\mathbb{Q}(\sqrt{5})$ &
  $1/\varphi \approx 0.618$ \\
$5$ & $2$ & $1+2i$ & $\mathbb{Q}(\sqrt{2})$ &
  $\sqrt{2}-1 \approx 0.414$ \\
$5$ & $4$ & $2+i$ & $\mathbb{Q}(\sqrt{5})$ &
  $1/\varphi^3 \approx 0.236$ \\
\end{tabular}}

\smallskip\noindent
The ``field'' column records $\mathbb{Q}(\sqrt{d}\,)$ where $d$ is the
square-free part of $(N\!+\!1)^2-\ell^2$.
The entry $(N,\ell)=(2,1)$ recovers the tame twindragon
(see below),
confirming that it belongs to the same two-parameter family.
The entry $(N,\ell)=(5,4)$ gives $\lambda_{5,4}=2+i$
--- the same contraction as the $2$-map tile in row~2 of
Table~\ref{tab:examples}, illustrating again
that AR depends only on $a = 1/\lambda$ (Corollary~\ref{cor:kappa}).

For fixed $\ell$, $\mathrm{AR}_{N,\ell}\to 0$ as $N\to\infty$
(stronger contraction elongates the attractor).
For fixed $N$, increasing $\ell$ toward $2\sqrt{N}$ drives $a$
toward the real axis and $\mathrm{AR}\to 0$;
setting $\ell=0$ gives $\lambda = i\sqrt{N}$ (purely imaginary~$a$)
and $\mathrm{AR} = 1/\sqrt{N}$, the maximum for a given~$N$.
For $N=2$ the IFS $\{-\tfrac{i}{\sqrt{2}}(z\pm 1)\}$
has attractor the rectangle $[-1,1]\times[-\sqrt{2},\sqrt{2}]$
(one verifies that $f_1$ and $f_2$ map it to the top and bottom halves);
this rectangle tiles the plane by translation.

\textit{Tame twindragon.}
The entry $(N,\ell)=(2,1)$ gives $\lambda=(1+i\sqrt{7})/2$
and the attractor is the
\emph{tame twindragon} \cite[p.~550]{Bandt1991}, the attractor of $\{f_{1,2}(z)=a(z\pm 1)\}$
with $a=(1-i\sqrt{7})/4$.
One computes $1-a^2=(11+i\sqrt{7})/8$, hence $|1-a^2|=\sqrt{2}$, and
Corollary~\ref{cor:kappa} gives
\[
  \mathrm{AR}^2 = \frac{\sqrt{2}-\tfrac{1}{2}}{\sqrt{2}+\tfrac{1}{2}}
               = \frac{2\sqrt{2}-1}{2\sqrt{2}+1}
               = \frac{9-4\sqrt{2}}{7}
               \approx 0.478,
  \qquad
  \mathrm{AR} \approx 0.691.
\]
This tile arises as the numeration digit set for
the complex base $\beta=(1+i\sqrt{7})/2$
(satisfying $\beta^2-\beta+2=0$ and $|\beta|^2=2$)
via $a=1/\beta$;\footnote{The attractor is the set
$\bigl\{\sum_{k\geq 1}d_k\beta^{-k}:d_k\in\{-1,+1\}\bigr\}\subset\mathbb{C}$.}
it is marked in Figure~\ref{fig:ar}.

The connection to metallic ratios in part~(iii) is not accidental.
The $m$-th metallic ratio $\mu_m = (m+\sqrt{m^2+4})/2$ is the
fundamental unit of $\mathbb{Z}[\mu_m]
\subset\mathbb{Q}(\sqrt{m^2+4})$.
In the $\ell=2$ sub-family, $\sqrt{m^2+4} = \sqrt{N+3}$
and the Cayley-type transform~\eqref{eq:field} maps
$(N\!+\!1)^2-4 = (N\!-\!1)(N\!+\!3)$ into $\mathbb{Q}(\sqrt{N+3})$.
For $N=2$ and $N=5$ the aspect ratio is an exact power of the
fundamental unit ($1/\varphi$ and $\sqrt{2}-1$ respectively);
for other~$N$ it remains in $\mathbb{Q}(\sqrt{N+3})$ but is not
generally a unit.

\begin{remark}
Both the aspect ratio and the principal-axis orientation of the
$\lambda_{N,\ell}$-family are governed by a single element of the
ring of integers: $\lambda^2-1 = \ell\lambda-(N\!+\!1) \in \mathbb{Z}[\lambda]$.
Its norm $(N\!+\!1)^2-\ell^2$ determines $\mathrm{AR}$
via~\eqref{eq:family-N}, and its argument determines~$\psi$
via~\eqref{eq:psi-N}.
At the boundary $\ell=0$ one has $\lambda^2-1 = -(N\!+\!1)$,
giving $\psi = -\pi/2$ (vertical major axis);
as $\ell\to 2\sqrt{N}$ the argument tends to~$0$ and
$\psi\to 0$ (horizontal, degenerate).
\end{remark}

\begin{figure}[ht]
\centering
\begin{tikzpicture}
\begin{axis}[
  width=8cm, height=5.5cm,
  xlabel={$\theta = |\arg(a)|$},
  ylabel={$\mathrm{AR}$},
  xmin=0, xmax=90,
  ymin=0, ymax=0.82,
  xtick={0,45,90},
  xticklabels={$0$, $\pi/4$, $\pi/2$},
  ytick={0, 0.6180, 0.7071},
  yticklabels={$0$, $1/\varphi$, $1/\sqrt{2}$},
  grid=major,
  grid style={dotted, gray!60},
  axis lines=left,
]
\addplot[blue!70!black, thick, smooth, domain=0:90, samples=200]
  { sqrt( (sqrt(1.25 - cos(2*x)) - 0.5) / (sqrt(1.25 - cos(2*x)) + 0.5) ) };
\addplot[red, mark=*, mark size=2pt, only marks] coordinates {(45, 0.6180)};
\addplot[red, mark=*, mark size=2pt, only marks] coordinates {(69.295, 0.6911)};
\addplot[red, mark=*, mark size=2pt, only marks] coordinates {(90, 0.7071)};
\node[anchor=north, font=\small] at (axis cs:45,0.610) {Twin Dragon};
\node[anchor=north, font=\small] at (axis cs:69.295,0.683) {tame twindragon};
\node[anchor=south east, font=\small] at (axis cs:88,0.717) {rectangle};
\addplot[dashed, thin, gray!80] coordinates {(0,0.7071)(90,0.7071)};
\addplot[dashed, thin, gray!80] coordinates {(45,0)(45,0.6180)};
\addplot[dashed, thin, gray!80] coordinates {(69.295,0)(69.295,0.6911)};
\addplot[dashed, thin, gray!80] coordinates {(90,0)(90,0.7071)};
\end{axis}
\end{tikzpicture}
\caption{Aspect ratio $\mathrm{AR}(\theta)$ of the attractor of
$\{a(z\pm 1)\}$ for $|a|=1/\sqrt{2}$, as a function of
$\theta = |\arg(a)|$.  The Twin Dragon ($\theta=\pi/4$),
the tame twindragon ($\theta=\arctan\sqrt{7}$), and
the rectangle attractor ($\theta=\pi/2$) are marked in red.}
\label{fig:ar}
\end{figure}
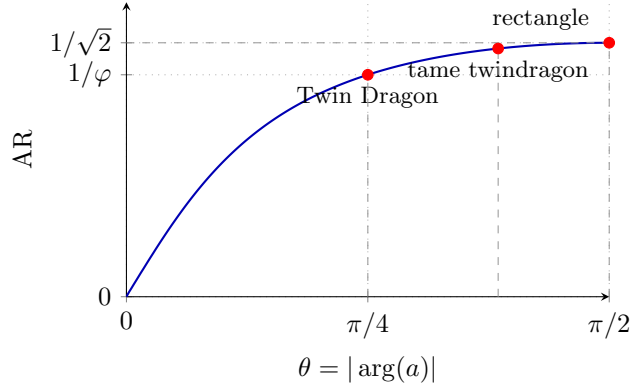

\section{Tiles over \texorpdfstring{$\mathbb{Z}[i]$}{ℤ[i]}}\label{sec:metallic}

In this section we specialise to tiles
over~$\mathbb{Z}[i]$: the contraction is $a = 1/\lambda$ for
a Gaussian integer~$\lambda$ and the translations form a
CRS modulo~$\lambda$.
Corollary~\ref{cor:kappa} then applies with $|a|^2=1/N$ and
$\kappa$ determined by the choice of CRS.

The dependence on $\kappa$ is already visible for a single
expansion factor.
Take $\lambda = 2+i$ ($N\!=\!5$).
The collinear digit set $\{0,1,2,3,4\}$ gives $\kappa=1$ and
$\mathrm{AR}=1/\varphi^3$ (row~$(5,4)$ of the table in
Section~\ref{sec:family-N}).
Replacing it by the two-dimensional CRS
$\{0,\,1,\,-i,\,i,\,1\!+\!i\}$
(which is still a complete residue system of $\mathbb{Z}[i]$
modulo~$2\!+\!i$, since $i\equiv 3\pmod{2\!+\!i}$)
yields $\sigma^2=4/5$, $\tau=(-8+6i)/25$, and $\kappa=1/2$.
Substituting into~\eqref{eq:family}:
\[
  \mathrm{AR}^2
  = \frac{\sqrt{5}-1}{\sqrt{5}+1}
  = \frac{3-\sqrt{5}}{2}
  = \frac{1}{\varphi^2},
  \qquad
  \mathrm{AR} = \frac{1}{\varphi}.
\]
Thus the same golden ratio governs this $5$-map tile with $\kappa=1/2$
as it does the Twin Dragon ($N\!=\!2$, $\kappa\!=\!1$),
while collinear digits for the same~$\lambda$ give the smaller
value~$1/\varphi^3$ (Figure~\ref{fig:kappa-tile}).
A natural question arises: for which~$\kappa$ does the aspect ratio
equal $1/\mu_m$, the reciprocal of the $m$-th metallic mean?

\begin{proposition}\label{prop:golden}
Let $\lambda = p+qi \in \mathbb{Z}[i]\setminus\mathbb{R}$ with
$|\lambda|^2 = N$, and let $m\geq 1$ be a positive integer.
The digit anisotropy required for $\mathrm{AR}=1/\mu_m$
\textup{(}the reciprocal $m$-th metallic mean
$\mu_m=(m+\sqrt{m^2+4})/2$\textup{)} is
\begin{equation}\label{eq:kappa-m}
  \kappa_m^*
  = \frac{m\,|\lambda^2-1|}{\sqrt{m^2+4}\,(N-1)},
\end{equation}
and $\kappa_m^*\leq 1$ if and only if $mq\leq N-1$.
In particular:
\begin{enumerate}[\upshape(a)]
\item For $m=1$ \textup{(}golden ratio\textup{)},
  $\kappa_1^*\leq 1$ for every~$\lambda$,
  with equality only for $\lambda=1+i$.
\item The collinear case $\kappa=1$ corresponds to
  $m=(N\!-\!1)/q$, provided this is an integer;
  thus the collinear aspect ratio of
  Corollary~\textup{\ref{cor:field}} equals
  $1/\mu_{(N-1)/q}$.
\item For $q=1$ \textup{(}i.e.\ $\lambda=p+i$\textup{)},
  every metallic ratio $\mu_1,\ldots,\mu_{p^2}$ is accessible.
\end{enumerate}
\end{proposition}

\begin{proof}
Setting $\mathrm{AR}^2 = 1/\mu_m^2$ in~\eqref{eq:family}
and solving for $\kappa$ gives
$\kappa = |1-a^2|\cdot m/\bigl((1-|a|^2)\sqrt{m^2+4}\bigr)$;
substituting $a=1/\lambda$, $|a|^2=1/N$,
$|1-a^2|=\sqrt{\mathbf{N}(\lambda^2-1)}/N$
yields~\eqref{eq:kappa-m}.

The condition $\kappa_m^*\leq 1$ becomes
$m^2|\lambda^2-1|^2 \leq (m^2+4)(N\!-\!1)^2$.
Since $|\lambda^2-1|^2 = (N\!-\!1)^2+4q^2$, this simplifies to
$4m^2q^2 \leq 4(N\!-\!1)^2$, i.e.\ $mq\leq N\!-\!1$.

\textit{(a)}: we must verify $q\leq N-1$ for every Gaussian integer
$\lambda=p+qi\notin\mathbb{R}$ with $|\lambda|>1$.
If $p\geq 1$ then $N=p^2+q^2\geq 1+q^2\geq q+1$, so $q\leq N-1$.
If $p=0$ then $|\lambda|>1$ forces $q\geq 2$, and
$N-1=q^2-1\geq q$ since $q\geq 2$.
Equality $q=N\!-\!1$ gives $p^2=1-q(q\!-\!1)$;
for $q=1$ this yields $p=1$, $\lambda=1+i$;
for $q\geq 2$ one has $p^2<0$, which is impossible.

\textit{(b)}: $\kappa_m^*=1$ iff $m^2q^2=(N\!-\!1)^2$,
so $m=(N\!-\!1)/q$.

\textit{(c)}: $q=1$ gives $mq = m \leq N-1 = p^2$.
\end{proof}

\begin{corollary}\label{cor:half}
Up to associates, $\kappa_1^*=1/2$ if and only if
$\lambda\in\{2+i,\;2+3i,\;1+4i\}$ \textup{(}$N=5,13,17$\textup{)}.
More generally, $\kappa_1^* = 1/m$ for an integer $m\geq 1$ only for
$m=1$ \textup{(}$\lambda=1+i$\textup{)} and $m=2$
\textup{(}the three primes above\textup{)}.
\end{corollary}

\begin{proof}
Write $\lambda=p+qi$, $N=p^2+q^2$.  Setting $\kappa_1^*=1/m$ in
$|\lambda^2-1|^2 = 5(N-1)^2/m^2$ and using
$|\lambda^2-1|^2 = (N-1)^2+4q^2$ gives
\[
  4q^2 = (N\!-\!1)^2\!\bigl(\tfrac{5}{m^2}-1\bigr).
\]
For $m=1$ this is $4q^2=4(N\!-\!1)^2$, so $q=N\!-\!1=1$, $N=2$:
only $\lambda=1+i$.
For $m=2$: $4q^2 = \frac14(N\!-\!1)^2$, so $N=4q+1$, $p^2=4q+1-q^2$,
and the positivity condition $p^2\geq 1$ becomes
$p^2+(q-2)^2 = 5$, i.e.\ a representation of~$5$ as a sum of two
squares.  The three solutions with $q\geq 1$, $p\geq 1$ are
$(p,q)=(2,1),(2,3),(1,4)$.
For $m\geq 3$: $5/m^2-1 < 0$, so $4q^2<0$ --- impossible.
\end{proof}

\begin{theorem}\label{thm:metallic}
For every integer $m\geq 1$:
\begin{enumerate}[\upshape(a)]
\item There exists a self-similar plane-filling tile
  in~$\mathbb{R}^2$ whose aspect ratio equals
  $1/\mu_m$, the reciprocal of the $m$-th metallic mean.
\item Any such tile over $\mathbb{Z}[i]$ requires at least
  $\lceil\sqrt{m}\,\rceil^2 + 1$ maps.
\item When $m$ is a perfect square, this lower bound is sharp:
  the minimum is $m+1$, achieved by the collinear CRS for
  $\lambda = \sqrt{m}+i$.
\end{enumerate}
\end{theorem}

\begin{proof}
\textit{(a):}
Let $\lambda = 1+mi \in \mathbb{Z}[i]$.
Then $N = |\lambda|^2 = m^2+1$ and $\operatorname{Im}(\lambda)=m$.
The standard digits $\{0,1,\ldots,m^2\}$ form a CRS
modulo~$\lambda$ (\S\ref{sec:family-N}); by
Lagarias--Wang~\cite{LagariasWang1996}, any CRS with an
expanding base $\lambda$ yields a self-affine tile of positive
Lebesgue measure whose translates by $\mathbb{Z}[i]$ tile
$\mathbb{R}^2$, and the open set condition
holds~\cite{Bandt1991}.
The collinear digit set has $\kappa = 1$, and
Proposition~\ref{prop:golden}(b) gives
$m = (N\!-\!1)/\operatorname{Im}(\lambda) = m^2/m$,
hence $\mathrm{AR} = 1/\mu_m$.

\textit{(b):}
Let $\lambda=p+qi\in\mathbb{Z}[i]\setminus\mathbb{R}$,
so $p\in\mathbb{Z}$ and $q\geq 1$ \textup{(}WLOG\textup{)}, and set
$k=\lceil\sqrt{m}\,\rceil$, so that
$(k-1)^2<m\leq k^2$.
The condition $mq\leq N-1$ with $N=p^2+q^2$ rearranges to
$p^2\geq mq-q^2+1$, and since $p^2$ is a non-negative integer,
$N = p^2+q^2 \geq \max\bigl(q^2,\,mq+1\bigr)$.
\emph{Case $q=1$}: $p^2\geq m$; since $p$ is an integer,
$p^2\geq k^2$, so $N=p^2+1\geq k^2+1$.
\emph{Case $q\geq 2$}: using $m\geq (k-1)^2+1$,
\[
  mq+1 \;\geq\; 2m+1 \;\geq\; 2\bigl((k-1)^2+1\bigr)+1
        \;=\; 2k^2-4k+5 \;\geq\; k^2+1,
\]
the last step being $(k-2)^2\geq 0$.
Hence $N\geq k^2+1 = \lceil\sqrt{m}\,\rceil^2+1$ in every case.

\textit{(c):}
When $m=k^2$, the choice $\lambda=k+i$ yields $N=k^2+1=m+1$ and
the collinear CRS has $\kappa=1$ with $(N\!-\!1)/q=m$.
\end{proof}

\begin{remark}
By the sum-of-two-squares theorem, the admissible norms $N=p^2+q^2$ with
$q\neq 0$ form the set $\{2,4,5,8,9,10,13,\ldots\}$.
\end{remark}

\begin{remark}
Item~(b) of Proposition~\ref{prop:golden} recovers every row of
Table~\ref{tab:examples} as a special case:
$\lambda=1\!+\!i$ gives $m=1$ and $\mathrm{AR}=1/\mu_1=1/\varphi$;
$\lambda=2\!+\!i$ gives $m=4$ and
$\mathrm{AR}=1/\mu_4=1/(2\!+\!\sqrt{5})=\sqrt{5}-2=1/\varphi^3$;
$\lambda=1\!+\!2i$ gives $m=2$ and $\mathrm{AR}=1/\mu_2=\sqrt{2}-1$.
Thus the collinear aspect ratios of Section~\ref{sec:family} and the
golden-ratio tiles of Table~\ref{tab:kappa} are part of a single
metallic-ratio spectrum parametrised by~$\kappa$.
Table~\ref{tab:kappa} lists CRS digit sets achieving $\kappa_1^*$
(i.e.\ $\mathrm{AR}=1/\varphi$) for small Gaussian primes~$\lambda$.

\begin{table}[ht]
\centering\small
\caption{Complete residue systems $D\subset\mathbb{Z}[i]$
(modulo~$\lambda$, centred to $\sum t_k=0$)
achieving $\mathrm{AR}=1/\varphi$.}
\label{tab:kappa}
\begin{tabular}{cccl}
$\lambda$ & $N$ & $\kappa^*$ & digit set $D$\\[2pt]\hline\\[-8pt]
$1+i$  & $2$  & $1$              & $\{0,\,1\}$ \\
$2+i$  & $5$  & $1/2$            & $\{0,\,1,\,-i,\,i,\,1\!+\!i\}$ \\
$1+2i$ & $5$  & $\sqrt{2/5}$     & $\{-1\!-\!2i,\,-1,\,i,\,1,\,1\!+\!i\}$ \\
$1+3i$ & $10$ & $\sqrt{13/45}$   & $\{-2,\,-2\!+\!2i,\,-1\!-\!i,\,-1,\,-1\!+\!i,\,-i,\,0,\,i,\,1\!-\!2i,\,1\}$ \\
$2+3i$ & $13$ & $1/2$            & $\{\!-\!2\!-\!3i,\,-2,\,-1\!\pm\!i,\,-1,\,\pm 2i,\,\pm i,\,1,\,1\!+\!i,\,2,\,3\!+\!2i\}$ \\
\end{tabular}
\end{table}

\noindent
Rows 3 and~4 have irrational $\kappa^*$ ($\sqrt{2/5}$ and $\sqrt{13/45}$),
yet Gaussian-integer digits realise these values exactly.
Corollary~\ref{cor:half} explains why rows 2 and~5 share
$\kappa^*\!=\!1/2$: only three Gaussian primes ($2\!+\!i$, $2\!+\!3i$, $1\!+\!4i$)
attain this value, and the proof reduces to the unique representation
$5=1^2+2^2$.
A CRS for the remaining prime $\lambda=1\!+\!4i$ ($N\!=\!17$)
achieving $\kappa\!=\!1/2$ also exists;
its $17$ digits are omitted from the table for space.
\end{remark}

\begin{figure}[ht]
\centering
\includegraphics[width=6cm]{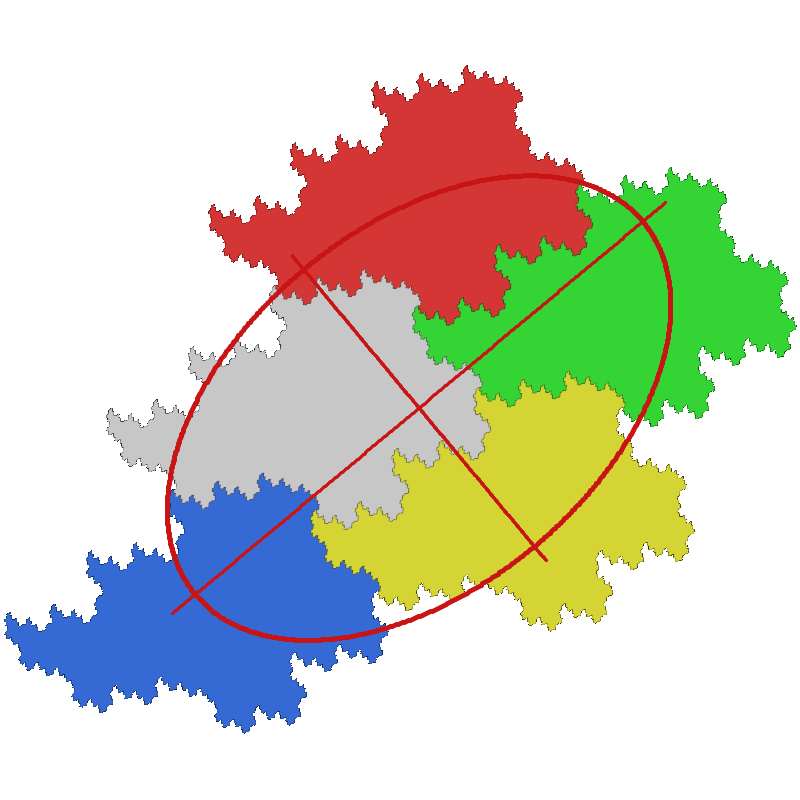}
\caption{The $5$-map tile with $\lambda=2+i$ and digit set
$\{0,1,-i,i,1\!+\!i\}$ ($\kappa=1/2$).
Its $1.5$-sigma principal-axis ellipse (red) confirms
$\mathrm{AR}=1/\varphi\approx 0.618$,
the same value as the Twin Dragon (Figure~\ref{fig:ellipses}a).}
\label{fig:kappa-tile}
\end{figure}

\begin{lemma}\label{lem:quadratic}
For $\lambda\in\mathbb{C}\setminus\mathbb{R}$ with $|\lambda|^2=N\geq 2$,
\begin{equation}\label{eq:imag-identity}
  |\lambda^2-1|^2 = (N-1)^2 + 4\,(\operatorname{Im}\lambda)^2.
\end{equation}
Consequently, the digit anisotropy~\eqref{eq:kappa-m} satisfies
$\kappa_m^*\leq 1$ if and only if
$m\leq (N\!-\!1)/|\operatorname{Im}\lambda|$.
The collinear case $\kappa=1$ corresponds to
$m=(N\!-\!1)/|\operatorname{Im}\lambda|$, which is a positive integer
only when $|\operatorname{Im}\lambda|$ divides $N\!-\!1$.
\end{lemma}

\begin{proof}
Write $\lambda=x+iy$,  so $N=x^2+y^2$ and
$\lambda^2-1=(x^2\!-\!y^2\!-\!1)+2ixy$.  Then
\[
  |\lambda^2-1|^2
  = (x^2\!-\!y^2\!-\!1)^2 + 4x^2y^2
  = (x^2\!+\!y^2\!-\!1)^2 + 4y^2
  = (N\!-\!1)^2 + 4y^2,
\]
since, by direct expansion,
$(x^2\!-\!y^2\!-\!1)^2+4x^2y^2
  - (x^2\!+\!y^2\!-\!1)^2 = 4y^2$.
Substituting into~\eqref{eq:kappa-m}:
$\kappa_m^*\leq 1$
iff $m^2\bigl((N\!-\!1)^2+4y^2\bigr)\leq (m^2\!+\!4)(N\!-\!1)^2$,
i.e.\ $4m^2y^2\leq 4(N\!-\!1)^2$, so $m\leq (N\!-\!1)/|y|$.
\end{proof}

\begin{corollary}\label{cor:unique-Zi}
Among all rings of integers of imaginary quadratic fields, $\mathbb{Z}[i]$
is the unique one where collinear self-similar tiles can have
metallic-ratio aspect ratios.
\end{corollary}

\begin{proof}
If $K\neq\mathbb{Q}(i)$, write $\mathcal{O}_K=\mathbb{Z}+\mathbb{Z}\omega$;
then $\operatorname{Im}\omega$ is either $\sqrt{D}$ or $\sqrt{D}/2$
for a square-free integer $D>1$.  Thus every non-real
$\lambda=a+b\omega\in\mathcal{O}_K$ has
$\operatorname{Im}\lambda=b\operatorname{Im}\omega\notin\mathbb{Q}$.
By Lemma~\ref{lem:quadratic}, the collinear metallic case would require
$m=(N\!-\!1)/|\operatorname{Im}\lambda|$ to be a positive integer,
which is impossible.  In $\mathbb{Z}[i]$, however,
$\operatorname{Im}\lambda = q\in\mathbb{Z}$,
and taking $q=1$ gives $m=(N\!-\!1)/1 = p^2$, always a positive integer.
\end{proof}

\section{Further directions}

An explicit closed-form construction of a CRS
achieving $\kappa=\kappa_m^*$ for all~$N$ would be desirable;
this amounts to a system of quadratic Diophantine equations
in the digit shifts $c_k\in\mathbb{Z}[i]$
(where $d_k = k + c_k\lambda$), which does not appear to simplify
into a universal formula.
More broadly, for a fixed expansion factor~$\lambda$,
characterising the set of all achievable anisotropy values~$\kappa$
--- and in particular the CRS that minimise or maximise~$\kappa$
--- is a natural combinatorial-geometric problem.

The equal-contraction assumption ($a_k \equiv a$) is essential to our
closed-form formulas.  For an IFS $\{a_k z + t_k\}$ with distinct
contraction ratios the fixed-point equation for the second moment
remains linear and is solvable in closed form for small~$N$;
extending the aspect-ratio classification to this setting is a natural
next step.
In a different direction, higher-order moments ($E[Z^3]$, $E[Z^4]$, \ldots)
satisfy analogous fixed-point equations and yield shape descriptors
beyond the aspect ratio (skewness, kurtosis);
their systematic study for self-similar measures remains open.

\end{document}